\documentclass[11pt, reqno]{amsart}
\usepackage{amsmath, amsthm, amscd, amsfonts, amssymb, graphicx, color}
\usepackage[bookmarksnumbered, colorlinks, plainpages]{hyperref}
\input{mathrsfs.sty}
\usepackage{hyperref}
\hypersetup{colorlinks=true, linkcolor=blue, anchorcolor=green,
citecolor=green, urlcolor=red, filecolor=magenta, pdftoolbar=true}

\newtheorem{theorem}{Theorem}[section]
\newtheorem{lemma}[theorem]{Lemma}

\newtheorem{corollary}[theorem]{Corollary}
\theoremstyle{definition}

\theoremstyle{remark}

 \def\AmSLaTeX{\leavevmode\hbox{$\cal A\kern-.2em\lower.376ex
 \hbox{$\cal M$}\kern-.2em\cal S$-\LaTeX}}
{\normalsize }
\def\u|{|\kern-0.1em|\kern-0.1em|}
\def\U|{\Big|\kern-0.1em\Big|\kern-0.1em\Big|}

\def\Tr{\hbox{Tr}\,}

\def\~{\hskip-2pt}

\def\<{\left\langle}
\def\>{\right\rangle}

\def\VV{\lower-0.1ex\hbox{$\ \begin{matrix}\vee\\[-2ex]\vee\end{matrix}\ $}}
\def\vv{\lower-0.2ex\hbox{$\ \begin{matrix}\wedge\\[-2ex]\wedge\end{matrix}\ $}}

\def\({\left(}
\def\){\right)}

\def\u|{|\kern-0.1em|\kern-0.1em|}
\def\U|{\Big|\kern-0.1em\Big|\kern-0.1em\Big|}

\numberwithin{equation}{section}
\def\s{\sharp}
\def\l{\lambda}

\def\Tr{\mathrm{Tr}}

\def\phi{\varphi}

\newcommand{\NORM}[1]{\left|\!\left|{#1}\right|\!\right|}

\begin{document}
\title[Jointly convex mappings ]{ Jointly  convex mappings related to the  Lieb's functional  and Minkowski type operator inequalities}

\author[M. Kian \MakeLowercase{and} Y. Seo ]{Mohsen Kian$^1$ \MakeLowercase{and} Yuki Seo$^2$ }

\address{$^1$ Mohsen Kian:\ \
Department of Mathematics,  University of Bojnord, P. O. Box
1339, Bojnord 94531, Iran}
\email{\textcolor[rgb]{0.00,0.00,0.84}{kian@ub.ac.ir}}

\address{$^2$ Yuki Seo:\ \
Department of Mathematics Education, Osaka Kyoiku University, Asahigaoka,Kashiwara, Osaka582-8582, Japan}
\email{\textcolor[rgb]{0.00,0.00,0.84}{yukis@cc.osaka-kyoiku.ac.jp}}



\begin{abstract}
Employing the notion of operator log-convexity, we study joint concavity$/$ convexity of multivariable operator functions: $(A,B)\mapsto F(A,B)=h\left[ \Phi(f(A))\ \sigma\ \Psi(g(B))\right]$, where $\Phi$ and $\Psi$ are positive linear maps and $\sigma$ is an operator mean. As applications, we prove jointly concavity$/$convexity of matrix trace functions $\Tr\left\{ F(A,B)\right\}$. Moreover, considering positive multi-linear mappings in $F(A,B)$, our study of the joint concavity$/$ convexity of $(A_1,\cdots,A_k)\mapsto h\left[ \Phi(f(A_1),\cdots,f(A_k))\right]$ provides some generalizations and complement to  results of Ando and Lieb concerning  the concavity$/$ convexity of maps involving tensor product. In addition, we present Minkowski type operator inequalities for a unial positive linear map, which is an operator version of Minkowski type matrix trace inequalities under a more general setting than Carlen and Lieb, Bekjan, and Ando and Hiai.

\end{abstract} \maketitle

\section{Introduction}
The notion of convexity in the non-commutative setting  arises in the context of $C^*$-algebras and matrix theory and receives notable attention for applications for example in quantum mechanics. As a typical non-commutative $C^*$-algebra, assume that $\mathcal{B}(\mathscr{H})$ is the algebra of all bounded linear operators on a complex Hilbert space $ \mathscr{H}$ and ${\Bbb M}_n:=\mathbb{M}_n(\mathbb{C})$ is the complex matrix algebra.  Several class of convex mappings on operators are known. Some of them, such as operator convex (concave) functions,   use the continuous functional calculus to induce an operator mapping on the set of all selfadjoint operators from a continuous real function.
    Operator convex (concave) functions were introduced by L\"{o}wner \cite{Low} and Kraus \cite{Kr} as non-commutative extensions of real convex (concave) functions  and later they were  characterized by Hansen and Pedersen \cite{H-P}. They presented a non-commutative extension of the well-known Jensen inequality  for operator convex functions.  As a related notion,
    operator log-convex functions first appeared in a paper of Aujla et al. \cite{Au} and then  characterized by Ando and Hiai \cite{AH}. An operator Jensen inequality for operator log-convex functions was shown in \cite{KD}.  In the case of several variable operator mappings, the convexity (concavity) of   $(A,B)\mapsto A^p\otimes B^q$ for proper real numbers $p,q$ was studied by Ando \cite{An}.
     Effros and Hansen \cite{E-H} studied the convexity of some two variable mapping $(A,B)\mapsto F(A,B):\mathcal{B}_h(\mathscr{H})\times\mathcal{B}_h(\mathscr{H})\to\mathcal{B}_h(\mathscr{H})$, 
 where $\mathcal{B}_h(\mathscr{H})$ is the set of all selfadjoint operators in $\mathcal{B}(\mathscr{H})$. Moreover, recently discussed in \cite{Eb,E,E-H,Mx-Ki}, operator perspectives  are two variable operator functions defined by $g(A,B)=B^{1/2}f(B^{-1/2}AB^{-1/2})B^{1/2}, A,B>0$ for every continuous function $f$ on $(0,\infty)$.  If $f$ is a positive operator monotone function with $f(1)=1$, then the operator perspective reduces to the operator mean with the representing function $f$ in \cite{Ku-An}. If $f$ is operator convex, then the operator perspective  has jointly convexity \cite{Eb,E} and are of significant application in quantum information in \cite{HM}.
 In \cite{Ha-Mx-Nj}, the authors considered some operator mappings  related to the Jensen inequality.

 Another class of useful  convex mappings on operators are convex functionals. The  convexity of  complex (real) valued mapping $X\mapsto \mathrm{Tr} f(X)$  on ${\Bbb M}_n$,
    where  $f$ is a one-variable real convex function, $X$ is a Hermitian matrix and $\mathrm{Tr}(\cdot)$ denotes the canonical trace,  is due to von Neumann \cite{v-N}. In 2003, Hansen and Pedersen \cite{H-P2} investigated the convexity of this mapping in the case where $f$ is a several variable function.  It has also been shown in \cite{AH} that  the mapping $X\mapsto\log \omega(f(X))$ is convex for every state $\omega$, where $f$ is operator log-convex. In order to generalize this, the authors of \cite{KY}  showed that this remains valid when $\omega$ and $\log t$ are replaced with a positive linear mapping and an operator concave function, respectively.

    In the study of entropy in quantum mechanics, the well-known Lieb's theorem   asserts that for every $0\leq p\leq1$ and every  matrix $K$, the mapping $(A,B)\mapsto \mathrm{Tr} K^*A^{p}KB^{1-p}$ is jointly concave on positive semidefinite matrices $A,B$. There have been many works  devoted to study the convexity (concavity) of trace functions regarding the extensions of Lieb's result. Carlen and Lieb \cite{Crl-Lb,Crl-Lb2} determined those $q,p>0$ for which    the trace functional $A\mapsto \mathrm{Tr}\left(B^*A^pB\right)^{q/p}$ is convex  (concave)  on positive semidefinite matrices, where $B$ is a fixed  matrix. \par
 We want to treat all these different situations in a more unified way.
 In the present paper, we consider operator mappings of the following types:
\begin{enumerate}
\item $F_1(A,B)=h[ \Phi(f(A))^{1/2}\Psi(g(B))\Phi(f(A))^{1/2}]$;
\item $F_2(A,B)=h[\Phi(f(A))\sigma \Psi(g(B))]$;
\item $F_3(A)=h[\Phi(f(A))]$.
\end{enumerate}
Here, the variables $A$ and $B$ are positive operators in $\mathcal{B}(\mathscr{H})$, $\Phi$ and $\Psi$ are positive linear maps between operator algebras, and $f,g,h$ are real valued continuous functions. Furthermore, $\sigma$ in (ii) is an operator mean in the Kubo-Ando sense \cite{Ku-An}.  We are mostly interested in the properties of the real valued functions $f,g,h$ for which the operator mapping $(A,B)\mapsto F(A,B)$ is jointly    convex (concave), and for which the matrix trace function $(A,B) \mapsto \Tr [F(A,B)]$ is jointly convex (concave). \par
 For instance, if $f(x)=x^p, g(x)=x^{1-p}, h(x)=x$, $\Phi(A)=K^*AK$, and $\Psi={\rm id}$, then the matrix function (i) under the trace implies the Lieb's theorem. If $f(x)=x^p, \Phi(A)=B^*AB$, and $h(x)=x^{q/p}$, then the matrix function (iii) under the trace implies Carlen and Lieb's result. If $f(x)=g(x)=x^p, h(x)=x^{1/p}$ and $\Phi = \Psi = {\rm id}$ and $\sigma$ is the arithmetic mean, then the matrix function (ii) under the trace implies the Minkowski type trace function. \par
 The paper is organized as follows: Section 2 is a preliminary.\par
 In Section 3, we treat the operator mapping  $F_2(A,B)$ and show its jointly  convexity$/$ concavity under suitable conditions on functions $f,g,h$. As applications, we prove jointly concavity$/$convexity of matrix trace functions $\Tr\left\{ F(A,B)\right\}$. Moreover, considering positive multi-linear mappings in $F(A,B)$, we study   the joint concavity$/$ convexity of $(A_1,\cdots,A_k)\mapsto h\left[ \Phi(f(A_1),\cdots,f(A_k))\right]$ as generalizations and complements to  results of Ando and Lieb  concerning  the concavity$/$ convexity of maps involving tensor product.

 In Section 4, we study the Minkowski type operator inequalities under a more general setting than Carlen and Lieb \cite{Crl-Lb}, Bekjan \cite{Bkj}, and Ando--Hiai \cite{AH}. As applications, we derive    Minkowski type matrix trace inequalities. Moreover, we give some estimations for the  operator determinants in the sense of Minkowski   operator  inequalities.


\section{Preliminaries}

 Throughout the paper, let $\mathcal{B}(\mathscr{H})$ denote the $C^*$-algebra of all bounded linear operators on a complex Hilbert space $(\mathscr{H},\<\cdot,\cdot\>)$ and let $I$ stand for the identity operator. We consider the usual L\"{o}wner partial order $\leq $ on the real space of self-adjoint operators. An operator $A$ in $\mathcal{B}(\mathscr{H})$ is said to be positive (denoted by $A\geq 0$) if $\< Ax,x\>\geq 0$ for all $x\in \mathscr{H}$. In particular, $A>0$ means that $A$ is positive and invertible.
  This turns the real subspace of self-adjoint operators into a partial ordered set, say $A\geq B$ if and only if $A-B$ is a positive operator. This is known as the L\"{o}wner partial order.  We denote by $\mathcal{B}(\mathscr{H})^{+}$ the set of all positive operators in $\mathcal{B}(\mathscr{H})$, and $\mathcal{B}(\mathscr{H})^{++}$ the set of all invertible $A\in \mathcal{B}(\mathscr{H})^{+}$.
A mapping $\Phi:\mathcal{B}(\mathscr{H})\to\mathcal{B}(\mathscr{K})$ is called positive if $\Phi(\mathcal{B}(\mathscr{H})^{+})\subseteq\mathcal{B}(\mathscr{K})^{+}$ and is called strictly positive if $\Phi(\mathcal{B}(\mathscr{H})^{++})\subseteq\mathcal{B}(\mathscr{K})^{++}$. It is called unital if $\Phi(I)=I$. 

\par

For a continuous real valued function $f:J\to\mathbb{R}$ and a  self-adjoint operator $A$ with spectrum contained in $J$, the self-adjoint operator $f(A)$ is defined by the continuous functional calculus. A continuous real valued function $f:J\to\mathbb{R}$ is said to be operator convex   if $f(\frac{A+B}{2})\leq \frac{f(A)+f(B)}{2}$ for all
 self-adjoint operators $A,B$ with spectra contained in $J$. If $-f$ is operator convex, then $f$ is called  operator concave.
 The function $f$ is said to be operator monotone (resp. decreasing) if $A\leq B$ implies $f(A)\leq f(B)$\, (resp. $f(A)\geq f(B)$) for  all self-adjoint operators $A,B$ with spectra contained in $J$.

  It is well-known that (see \cite{FMPS}) a continuous function $f:(0,\infty)\to (0,\infty)$ is operator concave if and only if $f$ is operator monotone. Typical examples of operator   convex functions are  $f(t)=t^p$, where  $p\in[-1,0]\cup[1,2]$.  In the case where $p\in[0,1]$,  the function $f(t)=t^p$  is operator concave and operator monotone.\par

An operator mean is a two-variable map $\sigma:\mathcal{B}(\mathscr{H})^{+}\times \mathcal{B}(\mathscr{H})^{+}\to\mathcal{B}(\mathscr{H})^{+}$ which satisfies the following properties:\\
\rm{(i)}\ joint monotonicity:   $A\leq C$ and $B\leq D$ implies $A\ \sigma \ B\leq C\ \sigma \ D$;\\
\rm{(ii)}\  upper continuity: if $A_n$ and $B_n$ are decreasing sequences of positive operators convergent to $A$ and $B$, respectively, in the strong operator topology, then $A_n\ \sigma \ B_n$ converges to $A\ \sigma \ B$. \\
\rm{(iii)}\  transformer inequality: $X^*(A\ \sigma \ B)X\leq (X^*AX)\ \sigma \ (X^*BX)$      for every $X$;\\
\rm{(iv)}\ normalization:  $I\ \sigma \ I=I$.

Known examples of operator means are operator arithmetic mean $A\nabla B=\frac{A+B}{2}$, operator geometric  mean $A\ \sharp \ B=A^{1/2}(A^{-1/2}BA^{-1/2})^{1/2}A^{1/2}$ and operator harmonic mean $A\ ! \ B=(A^{-1}\nabla B^{-1})^{-1}$ for $A,B\in \mathcal{B}(\mathscr{H})^{++}$.
The adjoint mean $\sigma^*$ of an operator mean $\sigma$ is defined by $A\ \sigma^* \ B=(A^{-1}\sigma B^{-1})^{-1}$.
By the Kubo-Ando theory \cite{Ku-An},   there exists a one-to-one correspondence   between operator means  and operator monotone functions $f:(0,\infty)\to (0,\infty)$ with $f(1)=1$, given by $A\ \sigma \ B=A^{1/2}f(A^{-1/2}BA^{-1/2})A^{1/2}$. The function $f$ is called the representing function of $\sigma$.  It is known that every operator mean $\sigma$ satisfies a monotonicity through every positive linear map $\Phi$, say
 \begin{align}\label{im}
\Phi(A\ \sigma \ B)\leq \Phi(A)\ \sigma \ \Phi(B)
 \end{align}
  for all $A,B\in\mathcal{B}(\mathscr{H})^{++}$.\par

 A continuous real valued function $f:(0,\infty)\to(0,\infty)$ is said to be operator log-convex (resp. operator log-concave) if
$f(A\nabla B)\leq f(A)\ \sharp \ f(B)$ (resp. $f(A\nabla B)\geq f(A)\ \sharp \ f(B)$) for all positive invertible operators $A,B$. It is known in \cite{AH} that $f$ is operator log-convex if and only if $f$ is operator monotone decreasing. In addition,  if $f$ is operator log-convex, then $f(A\nabla B)\leq f(A)\ \sigma \ f(B)$  for all positive invertible operators $A$ and $B$ and every symmetric operator mean $\sigma$.

Next, we recall the interpolation paths from \cite{FMPS}.  Let $A$ and $B$ be positive invertible operators in $\mathcal{B}(\mathscr{H})^{++}$. The interpolational paths are defined by
\[
A\ m_{r,t}\ B = A^{1/2}\left( (1-t)I+t(A^{-1/2}BA^{-1/2})^r\right)^{1/r}A^{1/2}
\]
for $r\in [-1,1]$ and $t\in [0,1]$. For each $t\in [0,1]$, $A\ m_{r,t}\ B$ is a path form $A!_t B$ to $A\nabla_t B$ via $A\s_t B$:
\begin{align*}
A\ m_{1,t}\ B & =A\ \nabla_t \ B = (1-t)A+tB;\\
A\ m_{0,t}\ B & = A\ \s_t\ B;\\
A\ m_{-1,t}\ B & = A\ !_t\ B = ((1-t)A^{-1}+tB^{-1})^{-1}.
\end{align*}
For each $t\in (0,1)$ the path $A\ m_{r,t}\ B$ is nondecreasing and norm continuous for $r\in {\Bbb R}$ and
\begin{align}\label{q-nq1}
A\ !_t\ B \leq A\ m_{r,t}\ B \leq A\ \nabla_t\ B
\end{align}
for $r\in [-1,1]$.\par
\medskip

\section{Jointly convex (concave) mappings related to the Lieb's functional}

 Every real function used in this section is assumed to be a non-negative real valued continuous function defined on $(0,\infty)$ (or more generally   on a subset of positive half line).
The next   lemma, which is the characterization of the operator log-convexity in terms of the interpolational paths,  is rather an straightforward corollary of \cite{AH}. For the reader's convenience, we give a short proof.

\begin{lemma} \label{th2-oplogconv}
Let $f$ be a nonnegative continuous  function on $(0,\infty)$. Then the following conditions are equivalent:
\begin{enumerate}
\item $f$ is operator monotone decreasing;
\item $f(A\ \nabla_t\ B)\leq f(A)\ m_{r,t}\ f(B)$ for all $A,B \in \mathcal{B}(\mathscr{H})^{++}$ and for all $r\in [-1,1]$ and $t\in [0,1]$;
\item $f$ is operator log-convex, i.e.,
\[
f(A\ \nabla_t\ B) \leq f(A)\ \s_t\ f(B) \ \mbox{for all $A,B \in \mathcal{B}(\mathscr{H})^{++}$ and for all $t\in [0,1]$;}
\]
\item $f(A\ \nabla_t \ B)\leq f(A)\ m_{r,t}\ f(B)$ for all $A,B \in \mathcal{B}(\mathscr{H})^{++}$ and for some $r\in [-1,1)$ and all $t\in [0,1]$.
\end{enumerate}
\end{lemma}

\begin{proof}
(i)$\Longrightarrow $(ii): We may assume that $f(x)>0$ for all $x\in (0,\infty)$. Since $1/f$ is positive and operator monotone on $(0,\infty)$, it follows that $1/f$ is operator concave on $(0,\infty)$. Hence
\[
f(A\ \nabla_t\ B)^{-1}\geq f(A)^{-1}\ \nabla_t\ f(B)^{-1}\quad \mbox{for all $t\in [0,1]$}
\]
so that
\[
f(A\ \nabla_t\ B)\leq f(A)\ !_t\ f(B) \quad \mbox{for all $t\in [0,1]$}
\]
which implies (ii) by noting \eqref{q-nq1}. \\
(ii)$\Longrightarrow $(iii):  If we put $r=0$ in (ii), then we have (iii).\\
(iii)$\Longrightarrow $(iv): It follows from $m_{0,t}=\s_t$.\\
(iv)$\Longrightarrow $(i): Put $t=1/2$ in (iv) so that $A\ m_{r,1/2}\ B\neq A\nabla B$ is a symmetric operator mean and
\[
f(A\ \nabla\ B) \leq f(A)\ m_{r,1/2}\ f(B).
\]
Hence $f$ is operator monotone decreasing by using \cite[Theorem 2.1]{AH}.
\end{proof}
Next  is a counterpart to Lemma~\ref{th2-oplogconv} for operator log-concave functions.
\begin{lemma}
Let $f$ be a nonnegative continuous  function on $(0,\infty)$. Then the following conditions are equivalent:
\begin{enumerate}
\item $f$ is operator monotone;
\item $f(A\ \nabla_t\ B)\geq f(A)\ m_{r,t}\ f(B)$ for all $A,B \in \mathcal{B}(\mathscr{H})^{++}$ and for all $r\in [-1,1]$ and all $t\in [0,1]$;
\item $f$ is operator log-concave, i.e.,
\[
f(A\ \nabla_t\ B) \geq f(A)\ \s_t\ f(B)\ \ \mbox{for all $A,B \in \mathcal{B}(\mathscr{H})^{++}$ and for all $t\in [0,1]$;}
\]
\item $f(A\ \nabla_t \ B)\geq f(A)\ m_{r,t}\ f(B)$ for all $A,B \in \mathcal{B}(\mathscr{H})^{++}$ and for some $r\in (-1,1]$ and all $t\in [0,1]$.
\end{enumerate}
\end{lemma}

\begin{proof} The implications $\mathrm{(i)}\Longrightarrow \mathrm{(ii)} \Longrightarrow \mathrm{(iii)} \Longrightarrow \mathrm{(iv)}$ holds obviously by noting the fact that the operator monotonicity of $f$ implies its operator concavity. \par
(iv)$\Longrightarrow $(i): If $t=1/2$, then $m_{r,1/2}$ is a symmetric mean for every $r\in (-1,1]$. Now if (iv)  holds, then there exists $r\in (-1,1]$ such that
\[
 f(A\ \nabla \ B)^{-1}\leq \left(f(A)\ m_{r,1/2}\ f(B)\right)^{-1} = f(A)^{-1}\ m^*_{r,1/2}\ f(B)^{-1}
\]
 Since $m^*_{r,1/2}=m_{-r,1/2}$ is symmetric and $-r\not= 1$, it follows from Lemma~\ref{th2-oplogconv} that $1/f$ is operator monotone decreasing and we have (i).
\end{proof}


\begin{lemma} \label{thm-2}
If $h$ is a nonnegative operator monotone function on $(0,\infty)$, then
\[
h(A\ !_t\ B) \leq h(A)\ !_t\ h(B)
\]
for all $A,B \in \mathcal{B}(\mathscr{H})^{+}$ and for all $t\in [0,1]$.
\end{lemma}
\begin{proof}
By continuity of the harmonic mean $!_t$, we may assume that $A$ and $B$ are positive invertible operators. Since $h$ is operator monotone,     $t\mapsto h(1/t)$ is operator monotone decreasing and so operator log-convex. Hence (ii) of Lemma~\ref{th2-oplogconv} yields
\[
h((A\ \nabla_t\ B)^{-1}) \leq h(A^{-1})\ !_t\ h(B^{-1})
\]
and so we have
\[
h(A!_t B) = h((A^{-1}\ \nabla_t\ B^{-1})^{-1}) \leq h(A)\ !_t\ h(B) \qquad \mbox{for all $t\in [0,1].$}
\]
\end{proof}
\par

 A two variable mapping  $F:\mathcal{B}(\mathscr{H})^{++} \times \mathcal{B}(\mathscr{H})^{++}\to\mathcal{B}(\mathscr{H})^{++}$  is called  jointly  convex  in $(A,B)$ if
\[
F((1-\l)A_1+\l A_2, (1-\l)B_1+\l B_2) \leq (1-\l)F(A_1,B_1)+\l F(A_2,B_2)
\]
for all  $A_i,B_i$ in $\mathcal{B}(\mathscr{H})^{++}$ for $i=1,2$ and for all $\l \in [0,1]$. The mapping  $F$ is called jointly   concave if     $-F$ is jointly   convex in $(A,B)$.  The mapping $F$ is called jointly   log-convex  (resp. jointly   log-concave) if
\[
F(A_1\nabla A_2, B_1\nabla B_2) \leq F(A_1,B_1)\ \s \ F(A_2,B_2)
\]
(resp.
\[
F(A_1\nabla A_2, B_1\nabla B_2) \geq F(A_1,B_1)\ \s \ F(A_2,B_2)
\]
) \\
for all  operators $A_i,B_i \in \mathcal{B}(\mathscr{H})^{++}$ for $i=1,2$.\par

In the next theorem we study jointly convexity$/$concavity of the operator mapping
\begin{align}\label{mee}
  F_2(A,B) = h(\Phi(f(A))\ \sigma \ \Psi(g(B)))
  \end{align}
in which $\Phi,\Psi:\mathcal{B}(\mathscr{H}) \to\mathcal{B}(\mathscr{K})$ are positive  linear maps and $\sigma$ is an operator mean.

\begin{theorem}\label{th3-mu}
  Let $\Phi,\Psi:\mathcal{B}(\mathscr{H}) \to\mathcal{B}(\mathscr{K})$ be positive  linear maps and let $\sigma$ be an operator mean. \\
  {\rm (i)}\ If $f,g$ are operator log-convex functions and $h$ is an operator monotone function,  then  \eqref{mee} is jointly   log-convex;\\
      {\rm (ii)}\ If $f,g$ are operator log-concave functions and $h$ is an operator monotone function, then  \eqref{mee} is jointly   log-concave.
\end{theorem}

\begin{proof}
First note that it follows from \cite[Theorem 4.8]{Ku-An} that
\begin{equation} \label{eq:hs}
[X!Y]\ \sigma \ [Z!W] \leq [X\sigma Z]\ ! \ [Y \sigma W]
\end{equation}
for all positive operators $X,Y,Z$ and $W \in \mathcal{B}(\mathscr{H})$. Indeed, since the adjoint operator mean $\sigma^*$ of $\sigma$ is jointly operator concave, we have
  \begin{align*}
   [X!Y]\ \sigma \ [Z!W]&=\left(X^{-1}\nabla Y^{-1}\right)^{-1}\sigma\left(Z^{-1}\nabla W^{-1}\right)^{-1}\nonumber\\
   &=\left[\left(X^{-1}\nabla Y^{-1}\right)\sigma^*\left(Z^{-1}\nabla W^{-1}\right)\right]^{-1}\\
   &\leq \left[ \left(X^{-1}\sigma^*Z^{-1}\right)\nabla \left(Y^{-1}\sigma^* W^{-1}\right)\right]^{-1}\\
   &=\left[ \left(X\sigma Z\right)^{-1}\nabla \left(Y\sigma  W\right)^{-1}\right]^{-1}=\left(X\sigma Z\right)\ ! \ \left(Y\sigma  W\right)\nonumber
  \end{align*}
  for all positive operators $X,Y,Z,W$. If $h$ is operator monotone, then the function  $t\mapsto h(1/t)$ is operator monotone decreasing and so is operator log-convex. Accordingly,  Lemma~\ref{thm-2} gives $h(M\ ! \ N)\leq h(M)\ ! \ h(N)$ for all positive operators $M,N \in \mathcal{B}(\mathscr{H})$. Therefore it follows from \eqref{eq:hs} that
  \begin{align}\label{qy1}
   h([X!Y]\ \sigma \ [Z!W])\leq h\left(\left(X\sigma Z\right)\ ! \ \left(Y\sigma  W\right)\right)\leq h\left(X\sigma Z\right)\ ! \ h\left(Y\sigma  W\right).
  \end{align}
   Now suppose that $A_1,A_2,B_1,B_2$ are positive invertible operators in $\mathcal{B}(\mathscr{H})^{++}$ and  $f$ and $g$ are operator log-convex functions.  Applying \eqref{im} and the operator  log-convexity of $f$ and $g$ we have
    \begin{align*}
     \Phi(f(A_1\nabla A_2)) \leq   \Phi(f(A_1))\  !  \ \Phi(f(A_2))\quad \mbox{and} \quad \Psi(g(B_1\nabla B_2))\leq \Psi(g(B_1))\ ! \ \Psi(g(B_2)).
  \end{align*}
Since every operator mean $\sigma$ is monotone in both variable, it gives
{\small  \begin{align}\label{qr1}
     \Phi(f(A_1\nabla A_2))\ \sigma \ \Psi(g(B_1\nabla B_2))\leq [\Phi(f(A_1))!\Phi(f(A_2))]\ \sigma \ [\Psi(g(B_1))!\Psi(g(B_2))].
  \end{align}}
 Since $h$ is operator monotone,  we conclude form \eqref{qy1} and  \eqref{qr1} that
 {\small  \begin{align}\label{qy2}
h(\Phi(f(A_1\nabla A_2))\ \sigma \ \Psi(g(B_1\nabla B_2)))
\leq h\left(\Phi(f(A_1)\ \sigma \ \Psi(g(B_1))\right)\ ! \ h\left(\Phi(f(A_2))\sigma  \Psi(g(B_2))\right)
\end{align}}
and this proves (i).  To prove (ii),  assume that   $f$ and $g$ are operator (log-)concave so that
\[
\Phi(f(A_1\nabla A_2)) \geq \Phi(f(A_1))\nabla \Phi(f(A_2))\quad \mbox{and} \quad \Psi(g(B_1\nabla B_2))\geq \Psi(g(B_1))\nabla \Psi(g(B_2)).
\]
The joint  monotonicity and concavity of the operator mean $\sigma$  ensure  that
\begin{align*}
\Phi(f(A_1\nabla A_2)) \sigma \Psi(g(B_1\nabla B_2)) & \geq [\Phi(f(A_1))\nabla \Phi(f(A_2))] \sigma [\Psi(g(B_1))\nabla \Psi(g(B_2))] \\
& \geq [\Phi(f(A_1))\sigma \Psi(g(B_1))] \nabla [\Phi(f(A_2))\sigma \Psi(g(B_2))].
\end{align*}
Since $h$ is operator monotone and operator concave, we have
\begin{align*}
h[\Phi(f(A_1\nabla A_2)) \sigma \Psi(g(B_1\nabla B_2))] & \geq
h[\Phi(f(A_1))\sigma \Psi(g(B_1))] \nabla [\Phi(f(A_2))\sigma \Psi(g(B_2))] \\
& \geq h[\Phi(f(A_1))\sigma \Psi(g(B_1))] \nabla h[\Phi(f(A_2))\sigma \Psi(g(B_2))]
\end{align*}
as desired.
\end{proof}
 Theorem~\ref{th3-mu} provides the following  slight improvement of \cite[Theorem 3.1]{KY}.
\begin{corollary} \label{cor-3-1}
Let $\Phi:\mathcal{B}(\mathscr{H}) \to\mathcal{B}(\mathscr{K})$ be a positive  linear map.  If $f:(0,\infty)\to(0,\infty)$ is an operator log-convex function and $h:(0,\infty)\to\mathbb{R}$ is an operator monotone function, then $F(A)= h(\Phi(f(A))$ is  log-convex.
\end{corollary}
It has been shown in \cite[Theorem 4.2]{KY} that if  $f,g$ are operator monotone decreasing functions and $h$ is operator monotone, then the functional $(A,B)\mapsto \mathrm{Tr}\left[F_1(A,B)\right]$ is separatly convex. The following corollary gives the separate convexity of the operator mapping $(A,B)\mapsto  F_1(A,B)$ without the presence of the trace functional.
\begin{corollary}
Let $\Phi, \Psi:\mathcal{B}(\mathscr{H}) \to\mathcal{B}(\mathscr{K})$ be positive  linear maps.  If $g:(0,\infty)\to(0,\infty)$ is an operator log-convex function and $h:(0,\infty)\to\mathbb{R}$ is an operator monotone function, then for a fixed positive operator  $A$
\[
F_1(A,B)= h[ \Phi(f(A))^{1/2}\Psi(g(B))\Phi(f(A))^{1/2}]
\]
is operator log-convex in the second term.
\end{corollary}

\begin{proof}
Put $\Gamma(X)=\Phi(f(A))^{1/2}\Psi(X)\Phi(f(A))^{1/2}$ and then $\Gamma$ is a positive linear map. The assertion then follows from Corollary~\ref{cor-3-1}.
\end{proof}

 It should be remarked that parts (i) and (ii) of Theorem \ref{th3-mu} do not remain valid if we replace operator log-convex functions $f$ and $g$ by operator convex functions. If fact, if $f$ and $g$ are operator convex and $h$ is operator monotone, the mapping \eqref{mee} does not even need to be convex. To see this, assume that $f(t)=g(t)=t^2$ and $h(t)=\sqrt{t}$. Consider $\Phi(A)=\Psi(A)=A$ and let $\sigma=\sharp$ be the operator geometric mean. If
 \begin{align*}
  A_1=\left[\begin{array}{cc}
     2 &1 \\ 1 &2   \end{array}\right],\quad  A_2=\left[\begin{array}{cc}
     1 &0 \\0 &2   \end{array}\right],\quad  B_1=\left[\begin{array}{cc}
     4 &-2 \\ -2 &3   \end{array}\right],\quad  B_2=\left[\begin{array}{cc}
     1 &-1 \\ -1 &3   \end{array}\right],
  \end{align*}
 Then
{\small \begin{align*}
   \left(\left(\frac{A_1+A_2}{2}\right)^2\sharp\left(\frac{B_1+B_2}{2}\right)^2\right)^{\frac{1}{2}}
   &=\left[\begin{array}{cc}
  1.7915   &  -0.3082 \\  -0.3082  &  2.1739 \end{array}\right]\\
&\nleqslant \left[\begin{array}{cc}
  1.6622 &  -0.3026 \\    -0.3026   & 2.1916 \end{array}\right]
 =  \frac{1}{2}\left(\left(A_1^2\sharp B_1^2\right)^{\frac{1}{2}}+\left(A_2^2\sharp B_2^2\right)^{\frac{1}{2}}\right).
 \end{align*}}
\par
 In the remainder of this section, we will pay attention to the case of  matrices. We denote by ${\Bbb M}_n={\Bbb M}_n(\mathbb{C})$ the algebra of all $n\times n$ matrices with complex entries.  We write ${\Bbb M}_n^+:=\{ A\in {\Bbb M}_n : A\geq 0\}$, the $n\times n$ positive semidefinite matrices, and ${\Bbb P}_n:=\{ A\in {\Bbb M}_n:A>0\}$, the $n\times n$ positive definite matrices. The usual trace on ${\Bbb M}_n$ is denoted by $\Tr$ and we will use the terms operator and matrix interchangeably. Every operator convex (operator monotone) function is a real convex (monotone increasing) function but the converse is not valid. However, it is known that \cite{H-P2} if $f$ is convex, then $\mathrm{Tr} f(\frac{A+B}{2})\leq \mathrm{Tr}\frac{f(A)+f(B)}{2}$.  Moreover,  if  $f$ is  convex, then there exist    unitaries $U$ and $V$ such that
 \begin{align}\label{con-re}
  f\left(\frac{A+B}{2}\right)\leq U^* \frac{f(A)+f(B)}{4} U+V^* \frac{f(A)+f(B)}{4} V.
 \end{align}
 If in addition  $f$ is  monotone increasing, then
  \begin{align}\label{con-re2}
  f\left(\frac{A+B}{2}\right)\leq U^* \frac{f(A)+f(B)}{2} U
 \end{align}
 for some unitary $U$.
Parallel to this, if $f$ is monotone increasing, then $A\leq B$ implies that $f(A)\leq U^*f(B)U$ for some unitary $U$. For a nice survey regarding operator inequalities for real convex functions  see \cite{BurLee}.\par

 Here, we consider jointly concavity$/$convexity of the trace function
\begin{equation} \label{eq-tr}
(A,B) \in {\Bbb M}_k \times {\Bbb M}_m \mapsto \Tr[F_2(A,B)]=\Tr[h[\Phi(f(A))\ \sigma \ \Psi(g(B))]]
\end{equation}
parallel to  \cite[Theorem 4.2]{KY}. This gives in addition a more general setting than \cite{Hi}.

\begin{theorem} \label{thm-3-M}
Let $\Phi: {\Bbb M}_k \mapsto {\Bbb M}_n$ and $\Psi : {\Bbb M}_m \mapsto {\Bbb M}_n$  be positive linear maps  and let $\sigma$ be an operator mean. Then
\begin{enumerate}
\item Let $f,g$ be operator log-convex functions.  If $h(x^{-1})^{-1}$ is monotone increasing and concave, then \eqref{eq-tr} is jointly log-convex.
     If $h(x)$ is monotone increasing and convex, then \eqref{eq-tr} is jointly convex. If $h(x)$ is monotone decreasing and concave, then \eqref{eq-tr} is  jointly concave.
\item Let $f,g$ be operator log-concave functions. If $h(x)$ is monotone increasing and concave,    then \eqref{eq-tr} is jointly concave.  If $h(x)$ is monotone decreasing and convex, then \eqref{eq-tr} is jointly convex. If $h(x^{-1})^{-1}$ is monotone decreasing and concave, then \eqref{eq-tr} is jointly log-convex.
\end{enumerate}
\end{theorem}
 Let use state some particular consequences  of Theorem \ref{thm-3-M}. It gives the
 joint convexity (concavity) of the mapping
 $$(A,B)\mapsto\mathrm{Tr}\left\{\Phi(A^p)\sigma\Psi(B^q)\right\}^s$$
for proper exponents $p,q,s$, see \cite[Lemma 3.3]{Hi}.

The function $t\mapsto 1/\log t$ is operator log-convex. Therefore the mapping
  $$(A,B)\mapsto\mathrm{Tr}\left\{\Phi((\log A)^{-1})\sigma\Psi((\log B)^{-1})\right\}^s$$
is jointly log-convex if $0\leq s\leq 1$ and is jointly convex if $s\geq1$. With $h(t)=\exp t$ we derive the joint convexity of
  $$(A,B)\mapsto\mathrm{Tr}\left\{\exp\left[\Phi((\log A)^{-1})\sigma\Psi((\log B)^{-1})\right]\right\}$$
  and
    $$(A,B)\mapsto\mathrm{Tr}\left\{\exp\left[\Phi(A^p)\sigma\Psi(B^q)\right]\right\}\quad (-1\leq p,q\leq0).$$
With $h(t)=\log t$, Theorem \ref{th3-mu} yields the joint concavity of
 $$(A,B)\mapsto\mathrm{Tr}\left\{\log\left[\Phi(A^p)\sigma\Psi(B^q)\right]\right\}\quad (0\leq p,q\leq1).$$
Now we give the proof of Theorem \ref{thm-3-M}.

\begin{proof}
 To prove (ii) we need    almost a  similar argument to (i). So, we only give the proof of part (i). Suppose that  $f,g$ are operator log-convex functions and  assume that $A_1,A_2\in{\Bbb M}_k^{++}$ and $B_1,B_2\in{\Bbb M}_m^{++}$.   It follows from \eqref{eq:hs} and \eqref{qr1} that
\begin{align}\label{eq-n1}
\Phi(f(A_1 \nabla A_2)) \sigma \Psi(g(B_1 \nabla B_2)) & \leq \left( \Phi(f(A_1)) ! \Phi(f(A_2)) \right) \sigma \left( \Psi(g(B_1)) ! \Psi(g(B_2))\right)\nonumber\\
& \leq \left(\Phi(f(A_1)) \sigma \Psi(g(B_1))\right) ! \left( \Phi(f(A_2)) \sigma \Psi(g(B_2)) \right).
\end{align}
If   $h(x^{-1})^{-1}$ is monotone increasing and concave, then  $h$ is monotone increasing and so there exists a unitary $V$ such that
{\small\begin{align}\label{q-nq4}
h\left[ \Phi(f(A_1 \nabla A_2))\ \sigma \ \Psi(g(B_1 \nabla B_2))\right]
& \leq V^* h\left[\left(\Phi(f(A_1))\ \sigma \ \Psi(g(B_1))\right) ! \left( \Phi(f(A_2))\ \sigma \ \Psi(g(B_2)) \right)\right]V.
\end{align}}
 On the other hand,  for all    positive definite matrices $M,N \in {\Bbb M}_n^{++}$, there exists a unitary $U$ such that
\[
h((M\nabla N)^{-1})^{-1} \geq U^*\left( h(M^{-1})^{-1}\nabla h(N^{-1})^{-1}\right)U
\]
and thus we have
\begin{align*}
h((M\nabla N)^{-1}) & \leq \left[ U^*\left( h(M^{-1})^{-1}\nabla h(N^{-1})^{-1}\right)U\right]^{-1} \\
& \leq U^*\left( h(M^{-1})^{-1}\nabla h(N^{-1})^{-1}\right)^{-1}U
\end{align*}
or equivalently
\begin{align}\label{q-nq2}
  h(M!N)\leq U^*[h(M)\ ! \ h(N)]U.
\end{align}
Accordingly we obtain
{\small\begin{align*}
& \Tr\left\{h[\Phi(f(A_1\nabla A_2))\sigma\Psi(g(B_1\nabla B_2))]\right\}\\
& \leq \Tr\left\{ h\left[\left(\Phi(f(A_1)) \sigma \Psi(g(B_1))\right) ! \left( \Phi(f(A_2)) \sigma \Psi(g(B_2)) \right)\right]\right\}\ \ \ \ \mbox{by \eqref{q-nq4}}\\
& \leq \Tr\left\{ h\left[\Phi(f(A_1)) \sigma \Psi(g(B_1))\right] ! h\left[ \Phi(f(A_2)) \sigma \Psi(g(B_2)) \right]\right\}\ \  \ \  \mbox{by \eqref{q-nq2}}\\
& \leq \Tr\left\{ h\left[\Phi(f(A_1)) \sigma \Psi(g(B_1)) \right] \right\} ! \Tr\left\{h\left[ \Phi(f(A_2)) \sigma \Psi(g(B_2))\right]\right\},
\end{align*}}
where the last inequality follows from applying \eqref{im} to $\Tr$. This ensures that if $h(x^{-1})^{-1}$ is monotone increasing and concave,  then \eqref{eq-tr} is jointly log-convex.

Now let $h$ be monotone increasing and convex. From \eqref{q-nq4} we learn that
 {\small\begin{align*}
\Tr\left\{h[\Phi(f(A_1\nabla A_2))\sigma\Psi(g(B_1\nabla B_2))]\right\}
&\leq \Tr\left\{h\left[\left(\Phi(f(A_1))\sigma \Psi(g(B_1))\right)!\left(\Phi(f(A_2))\sigma  \Psi(g(B_2))\right)\right]\right\} \\
&\leq \Tr\left\{h\left[\left(\Phi(f(A_1))\sigma \Psi(g(B_1))\right)\ \nabla \   \left(\Phi(f(A_2))\sigma  \Psi(g(B_2))\right)\right]\right\}\\
&\leq \Tr\left\{h\left[\Phi(f(A_1))\sigma \Psi(g(B_1))\right]\right\}\ \nabla \ \Tr\left\{h\left[\Phi(f(A_2))\sigma  \Psi(g(B_2))\right]\right\},
\end{align*}}
in which
 we utilized  the  operator arithmetic-harmonic mean inequality in the second inequality and apply \eqref{con-re2} for the convexity of $h$ to derive the last inequality. Hence \eqref{eq-tr} is jointly convex, when $h$ is  monotone increasing and convex.

 Next suppose that  $h$ is  monotone decreasing and concave.  From \eqref{eq-n1} we find a unitary $V$ such that
 {\small\begin{align}\label{q-nq4}
h\left[ \Phi(f(A_1 \nabla A_2))\ \sigma \ \Psi(g(B_1 \nabla B_2))\right]
& \geq V^* h\left[\left(\Phi(f(A_1))\ \sigma \ \Psi(g(B_1))\right)\nabla \left( \Phi(f(A_2))\ \sigma \ \Psi(g(B_2)) \right)\right]V,
\end{align}}
  where we use the operator arithmetic-harmonic mean inequality. For every pair of  positive definite matrices $X,Y\in\mathbb{M}_n^{++}$,  by the concavity of $h$ we can apply \eqref{con-re} to find unitaries $U$ and $W$  such that
  \begin{align*}
   h(X\nabla Y)\geq U^*\frac{h(X)+h(Y)}{4}U+W^*\frac{h(X)+h(Y)}{4}W.
  \end{align*}
  Therefore,
  \begin{align*}
\Tr\left\{h[\Phi(f(A_1\nabla A_2))\sigma\Psi(g(B_1\nabla B_2))]\right\}\geq \Tr\left\{h(X\nabla Y)\right\}\geq  \Tr\left\{h(X)\nabla h(Y)\right\}
  \end{align*}
  with $X=\Phi(f(A_1)) \sigma \Psi(g(B_1))$ and $Y=\Phi(f(A_2)) \sigma \Psi(g(B_2))$. This implies that \eqref{eq-tr} is jointly concave, when $h$ is  monotone decreasing and concave. The proof of (i) is now completed.
\end{proof}
  As a particular case of Theorem \ref{thm-3-M}, the next corollary implies the convexity of
 \begin{align}\label{ks}
\mathbb{M}_k^{++}\to\mathbb{C}:\quad A\mapsto\Tr[\Phi(A^p)^{-1/p}].
\end{align}
  We note that the convexity of \eqref{ks} implies the joint convexity of $(A,B)\mapsto(A^p+B^p)^{-1/p}$.  In this direction, it was shown in \cite{Crl-Lb2,Hi} that the mapping
  $$\mathbb{M}_n^+\times\mathbb{M}_n^+\to\mathbb{C}:\ \ (A,B)\mapsto(A^p+B^p)^{1/p}$$
     is jointly convex if and only if $1\leq p\leq 2$.
\begin{corollary} \label{cor-F3}
Let $\Phi : {\Bbb M}_k \mapsto {\Bbb M}_n$ be a unital positive linear map. Then mapping \eqref{ks}
 is convex for all $p\in[-1,1]\backslash\{0\}$.
\end{corollary}
\begin{proof}
In the case of $0<p\leq 1$, $f(x)=x^p$ is operator log-concave and $h(x)=x^{-1/p}$ is monotone decreasing and convex. Hence part (ii) of Theorem \ref{thm-3-M} gives the convexity of  \eqref{ks}. In the case of $-1\leq p<0$, $f(x)=x^p$ is operator log-convex and $h(x)=x^{-1/p}$ is monotone increasing and convex. In this case we use part (i) of Theorem \ref{thm-3-M}.
\end{proof}
\par\medskip


The next theorem gives a complementary result to \cite[Theorem 4.2]{KY}.
\begin{theorem}\label{th-n}
Let $\Phi:{\Bbb M}_k\to{\Bbb M}_n$ and $\Psi:{\Bbb M}_m\to{\Bbb M}_n$  be unital positive linear  maps and let $K\in{\Bbb M}_n$.
If $f_1,f_2$  are operator monotone functions, then \\
{\rm (i)}\ the mapping
 {\small  \begin{align*}
    \mathbb{M}_k^{++}\times\mathbb{M}_m^{++}\to\mathbb{C}:(A,B) \mapsto\Tr\left\{\Phi(f_1(A))^pK^*\Psi(f_2(B))^{1-p}K\right\}
      \end{align*}}
is jointly concave for all $0\leq p\leq 1$; \\
{\rm (ii)}\ the mapping
 {\small  \begin{align*}
    \mathbb{M}_k^{++}\times\mathbb{M}_m^{++}\to\mathbb{C}:(A,B) \mapsto\Tr\left\{\Phi(f_1(A))^pK^*\Psi(f_2(B))^{-1-p}K\right\}
      \end{align*}}
is jointly convex for all $-1\leq p\leq 0$.
\end{theorem}
\begin{proof}
Suppose that  $f(x)=x^p$  and  $g(x)=x^{1-p}$  for $p\in[0,1]$. By  setting $\Phi(A)=A$ and $\Psi(A)=K^*AK$ for a fixed matrix $K\in\mathbb{M}_n$ and $h(x)=x$, \cite[Theorem 2.1]{Hi} implies that  $\Tr\left\{F_1(A,B)\right\}$ is jointly concave. In other words,
{\small \begin{align}\label{qee2}
\Tr\left\{(X\nabla Y)^pK^*(Z\nabla W)^{1-p}K\right\}\geq
\Tr\left\{X^pK^*Z^{1-p}K\right\}\nabla \Tr\left\{Y^pK^*W^{1-p}K\right\}
  \end{align}}
for all $X,Y,Z,W\in{\Bbb M}_n^{++}$.
Then assume that  $A_1,A_2\in{\Bbb M}_k^{++}$ and $B_1,B_2\in{\Bbb M}_m^{++}$ and let  $f_1$ and $f_2$ be operator concave functions.  For every $0\leq p\leq1$, the operator monotonicity of $t\mapsto t^p$ and $t\mapsto t^{1-p}$ gives
    \begin{align}\label{qp1}
     \Phi(f_1(A_1\nabla A_2))^p \geq   [\Phi(f_1(A_1))  \nabla  \Phi(f_1(A_2))]^p
     \end{align}
     and
         \begin{align}\label{qp2}
 \Psi(f_2(B_1\nabla B_2))^{1-p}\geq [\Psi(f_2(B_1))\nabla\Psi(f_2(B_2))]^{1-p}.
  \end{align}
  It follows from \eqref{qp1} and \eqref{qp2} that
    {\small\begin{align}\label{qee11}
       &\Tr\left\{\Phi\left(f_1(A_1\nabla A_2)\right)^pK^*\Psi\left(f_2(B_1\nabla B_2)\right)^{1-p}K\right\}\geq \Tr\left\{(X\nabla Y)^pK^*(Z\nabla W)^{1-p}K\right\}
     \end{align}}
  in which we use the   notation
    {\small  \begin{align*}
   X=\Phi(f_1(A_1)),\quad Y=\Phi(f_1(A_2)),\quad Z=\Psi(f_2(B_1)),\quad W=\Psi(f_2(B_2))
  \end{align*}}
for brief. We conclude from \eqref{qee2} and \eqref{qee11} that the trace functional
    {\small  \begin{align}\label{q-kn1}
    \mathbb{M}_k^{++}\times\mathbb{M}_m^{++}\to\mathbb{C}:(A,B) \mapsto\Tr\left\{\Phi(f_1(A))^pK^*\Psi(f_2(B))^{1-p}K\right\}
      \end{align}}
is jointly concave for all $0\leq p\leq 1$, which gives {\rm (i)}.
 Next suppose that $-1\leq p\leq 0$ so that      $t\mapsto t^{p}$ and $t\mapsto t^{-1-p}$ are operator decreasing and we obtain
   \begin{align}\label{qpp1}
     \Phi(f_1(A_1\nabla A_2))^p \leq   [\Phi(f_1(A_1))  \nabla  \Phi(f_1(A_2))]^p
     \end{align}
     and
         \begin{align}\label{qpp2}
 \Psi(f_2(B_1\nabla B_2))^{-1-p}\leq [\Psi(f_2(B_1))\nabla\Psi(f_2(B_2))]^{-1-p}.
  \end{align}

   We conclude  from \eqref{qpp1} and \eqref{qpp2} that
    {\small\begin{align}\label{qee1}
\Tr\left\{\Phi\left(f_1(A_1\nabla A_2)\right)^pK^*\Psi\left(f_2(B_1\nabla B_2)\right)^{-1-p}K\right\} \leq \Tr\left\{(X\nabla Y)^pK^*(Z\nabla W)^{-1-p}K\right\}
     \end{align}}
     with the same notation for $X,Y,Z,W$ as in part (i).
  Now we apply \cite[Theorem 2.1]{Hi} with setting $\Phi(A)=A, \Psi(A)=K^*AK, f(x)=x^p, g(x)=x^{-1-p}$ and $h(x)=x$ to obtain {\rm (ii)}.
\end{proof}

 Next, we consider the jointly convex mappings   involving   multi-linear mappings. Our motivation is a result of     Lieb  (see e.g. \cite{Bh}),  which states that $(A,B)\mapsto A^p\otimes B^{1-p}$ is jointly   concave for every $p\in[0,1]$. In fact,   tensor product mapping $(A,B)\mapsto A^p\otimes B^{1-p}$  is a particular example of the map $(A,B)\mapsto \Phi(A^p, B^{1-p})$ for a bilinear map $\Phi$. So, it is natural to  study jointly convex mappings involving positive multi-linear mappings.  Recall that a mapping $\Phi:{\Bbb M}_n^k\to {\Bbb M}_m$ is called multi-linear  if it is linear in each of its variables.  $\Phi$  is called positive if $\Phi(A_1,\cdots,A_k)\in {\Bbb M}_m^{+}$, when $A_i\in {\Bbb M}_n^{+}$ for all $i=1,\ldots,k$. It is  called unital if $\Phi(I,\cdots,I)=I$.
 For  every positive linear map $\Psi:\mathbb{M}_{n^k}\to\mathbb{M}_m$, the mapping $\Phi:\mathbb{M}_n^k\to\mathbb{M}_m$ defined by $\Phi(A_1,\cdots,A_k)=\Psi(A_1\otimes\cdots\otimes A_k)$ is positive and multi-linear. In particular, $(A_1,\cdots,A_k)\mapsto A_1\otimes\cdots\otimes A_k$  and $(A_1,\cdots,A_k) \mapsto A_1\circ\cdots\circ A_k$ are   positive multi-linear mappings, in which $X\circ Y$ is the Hadamard product of $X$ and $Y$. To see more examples and information about  positive multi-linear mappings, the authors can refer to \cite{DKS}.

Let $\Phi:{\Bbb M}_n^k\to {\Bbb M}_m$ be a positive multilinear mapping. We here study  the joint convexity  of  the mapping
\begin{align}\label{me}
(A_1,\cdots,A_k)\mapsto\Tr\left\{h\left(\Phi(f_1(A_1),\cdots,f_k(A_k))\right)\right\},
\end{align}
where $(A_1,\cdots,A_k)$ is a $k$-tuple of positive definite matrices  on ${\Bbb M}_n$.\par
  We need the following  multi-variable extension of \eqref{im}.
 \begin{lemma}\cite[Proposition 3.6]{DKS}\label{dks}
 Let  $\Phi:{\Bbb M}_n^k\to {\Bbb M}_m$  be a strictly positive multilinear mapping. If  $\sigma$ is an operator  mean with a super-multiplicative representing function, then
\begin{equation*}
  \Phi(A_1\sigma B_1,\cdots,A_k\sigma B_k)\leq\Phi(A_1,\cdots,A_k)\ \sigma \ \Phi(B_1,\cdots,B_k)
\end{equation*}
for all positive definite matrices $A_i,B_i$,\ $(i=1,\cdots,k)$.
 \end{lemma}


\begin{theorem}\label{th4-mu}
  Let $\Phi:{\Bbb M}_n^k\to {\Bbb M}_m$ be a positive multilinear mapping  and let  $f_1,\ldots,f_k$ be operator log-convex functions.\\
    {\rm (i)}\ The mapping
    \begin{align}\label{me1}
(A_1,\cdots,A_k)\mapsto\Phi(f_1(A_1),\cdots,f_k(A_k))
\end{align}
is jointly  log-convex.\\
 {\rm (ii)}\ The mapping
    \begin{align}\label{me1}
(A_1,\cdots,A_k)\mapsto\Phi(f_1(A_1),\cdots,f_k(A_k))^{-1}
\end{align}
is jointly  log-concave.\\
    {\rm (iii)}\   The mapping \eqref{me} is jointly convex for every convex monotone increasing function $h$.
    \end{theorem}
\begin{proof}
Suppose that $A_1,\cdots,A_k$ and $B_1,\cdots,B_k$ are $k$-tuples of positive definite matrices in ${\Bbb M}_n$.  For every $i=1,\cdots,k$,  it follows from the operator log-convexity of $f_i$  that
\begin{align*}
  f_i(A_{i}\nabla B_{i})\leq f_i(A_{i})\ \sharp \ f_i(B_{i})\qquad (i=1,\cdots,k),
  \end{align*}
  whence we have
\begin{align}\label{qwwq1}
  \Phi\left(f_1(A_{1}\nabla B_{1}),\cdots,f_k(A_{k}\nabla B_{k})\right)\leq \Phi\left(f_1(A_{1})\ \sharp \ f_1(B_{1}),\cdots,f_k(A_{k})\ \sharp \ f_k(B_{k})\right)
  \end{align}
  from monotonicity of $\Phi$.
The representing function of the operator geometric mean is super-multiplicative. Accordingly   Lemma \ref{dks} can be applied to write
{\small\begin{align}\label{qw1}
  \Phi\left(f_1(A_{1})\sharp f_1(B_{1}),\cdots,f_k(A_{k})\sharp f_k(B_{k})\right)\leq
  \Phi\left(f_1(A_{1}),\cdots,f_k(A_{k})\right)\ \sharp \ \Phi\left(f_1(B_{1}),\cdots,f_k(B_{k})\right).
  \end{align}}
  Part (i) now follows from \eqref{qwwq1} and \eqref{qw1}. Moreover, noting that $(X\sharp Y)^{-1}=X^{-1}\sharp Y^{-1}$, we derive (ii).\par
For (iii), since $h$ is monotone increasing, it follows from  \eqref{qwwq1} and \eqref{qw1} that there exist two unitaries $U$ and $V$ such that
  $$h\left(\Phi\left(f_1(A_{1}\nabla B_{1}),\cdots, f_k(A_{k}\nabla B_{k})\right)\right)\leq U^*h\left(X\sharp Y\right)U\leq U^*V^*h(X\nabla Y)VU,$$
  where we use $X= \Phi\left(f_1(A_{1}),\cdots,f_k(A_{k})\right)$ and $Y=\Phi\left(f_1(B_{1}),\cdots,f_k(B_{k})\right)$ for short. Furthermore, the convexity of  $h$ guarantees the existence of a unitary $W$ such that $h(X\nabla Y)\leq W^*[h(X)\nabla h(Y)]W$. Hence
 {\small  \begin{align*}
\Tr\left\{h\left(\Phi\left(f_1(A_{1}\nabla B_{1}),\cdots, f_k(A_{k}\nabla B_{k})\right)\right)\right\}&\leq\Tr\left\{U^*V^*W^* [h(X)\nabla h(Y)]WVU\right\}\\
&= \Tr\left\{h\left(X\right)\right\}\nabla \ \Tr\left\{h\left(Y\right)\right\},
\end{align*}}
and this completes the proof of (iii).
\end{proof}
Here are some particular consequences of Theorem \ref{th4-mu}.
\begin{corollary}

\textbf{1.}\ Considering $f_1(x)=x^p$ and $f_2(x)=x^q$  shows that
 $$(A,B)\mapsto (A^p\otimes B^q)$$
  is jointly log-convex, when $p,q\in[-1,0]$ and is  jointly (log-)concave, when $p,q\in[0,1]$. In particular, we derive the Lieb's result: the mapping $(A,B)\mapsto A^p\otimes B^{1-p}$ is jointly concave for every $p\in[0,1]$.
 \\
\textbf{2.}\
Let  $\Phi:\mathbb{M}_n\to \mathbb{M}_m$ be a positive linear map. The trace functional  \begin{align}\label{qt1}
(A_1,\ldots,A_k)\mapsto
\mathrm{Tr} \ \Phi\left(A_1^{p_1}\otimes\ldots\otimes A_k^{p_k}\right)^s
\end{align}
is jointly convex, when $-1\leq p_i\leq0$ and $s\geq1$. In particular,
\begin{align}\label{qt1}
(A_1,\ldots,A_k)\mapsto
\mathrm{Tr} \   A_1^{r_1}\otimes\ldots\otimes A_k^{r_k}
\end{align}
is jointly convex for all $r_i\leq 0$, \ $(i=1,\ldots,k)$. \\
\textbf{3.}\ Assume that   $\Phi_i:\mathbb{M}_n\to \mathbb{M}_m$,\, $(i=1,\ldots,k)$ are positive linear mappings. Considering the positive multilinear mapping  $\Phi:\mathbb{M}_n^k\to \mathbb{M}_{m^k}$ defined by $\Phi(A_1,\ldots,A_k)=\Phi_1(A_1)\otimes\cdots\otimes\Phi_k(A_k)$, we obtain the joint log-convexity of
$$(A_1,\ldots,A_k)\mapsto \prod_{i=1}^k\bigotimes \Phi_i(A_i^{p_i}),$$
where $p_i\in[-1,0]$\, $(i=1,\ldots,k)$
and the joint log-concavity of
  $$(A_1,\ldots,A_k)\mapsto \prod_{i=1}^k\bigotimes \Phi_i(A_i^{p_i})^{-1}.$$
  This gives  \cite[Corollary 5.1]{An} and a complementary result to \cite[Theorem 5]{An}. Moreover, for every $s\geq1$, the mapping
 $$(A_1,\ldots,A_k)\mapsto \mathrm{Tr} \left\{\prod_{i=1}^k\bigotimes \Phi_i(A_i^{p_i})^s\right\} $$
is jointly convex.\\
\textbf{4.}\ Consider  the  monotone increasing convex function $h(t)=\exp t$  concludes the joint convexity of
 \begin{align*}
(A_1,\ldots,A_k)\mapsto
\mathrm{Tr}\left\{\exp\left[\Phi_1(A_1^{p_1})\otimes\ldots\otimes\Phi_k(A_k^{p_k})\right]\right\}
\end{align*}
 when $-1\leq p_i\leq0$.

\end{corollary}
\medskip

  \section{Minkowski type operator inequalities}

If $a_1,\ldots,a_n$ and $b_1,\ldots, b_n$ are  positive real numbers, then the Minkowski inequality asserts  that
\begin{equation} \label{eq:M1}
 \left( \sum_{i=1}^n (a_i+b_i)^p \right)^{\frac{1}{p}} \leq \left( \sum_{i=1}^n a_i^p\right)^{\frac{1}{p}} + \left( \sum_{i=1}^n b_i^p\right)^{\frac{1}{p}}\qquad \mbox{for $p\geq 1$}
\end{equation}
and
\begin{equation} \label{eq:M2}
 \left( \sum_{i=1}^n (a_i+b_i)^p \right)^{\frac{1}{p}} \geq \left( \sum_{i=1}^n a_i^p\right)^{\frac{1}{p}} + \left( \sum_{i=1}^n b_i^p\right)^{\frac{1}{p}}\qquad \mbox{for $p<0$ \ or \ $0<p<1$}.
\end{equation}
 The Minkowski inequality is one of the most fundamental inequalities in functional analysis. As a trace version of the Minkowski inequality, Carlen and Lieb showed that  the mapping
\[
\Phi_p(A_1,\ldots,A_k) = \Tr\left[ (A_1^p+\cdots +A_k^p)^{1/p}\right],\qquad (A_i\in\mathbb{M}_n^{++})
\]
is jointly concave for every $0<p\leq 1$.
The concavity of $\Phi_p$  results to the Minkowski trace inequality
\[ \Tr\left[\left( \sum_{i=1}^n (A_i+B_i)^p\right)^{1/p}\right] \geq \Tr\left[\left(\sum_{i=1}^n A_i^p\right)^{1/p}\right] + \Tr\left[\left(\sum_{i=1}^n B_i^p\right)^{1/p}\right]. \]

We want to consider the operator versions of \eqref{eq:M1} and \eqref{eq:M2}. Before that, we note that  Corollary~\ref{cor-F3} gives the following deformed Minkowski type trace inequalities:
\begin{corollary}
Let $A$ and $B$ be positive definite matrices, and $\Phi$ be a unital positive linear map. Then
\[
\Tr\left[ \Phi((A+B)^p)^{-1/p}\right] \leq \Tr\left[ \Phi(A^p)^{-1/p}\right]+\Tr\left[ \Phi(B^p)^{-1/p}\right]
\]
for all $p\in[-1,1]\backslash\{0\}$.
\end{corollary}

Now, from the viewpoint of operator inequalities, we would expect the following Minkowski inequality for operators: For positive invertible operators $A_1,\ldots,A_n$ and $B_1,\ldots,B_n$,
\begin{equation} \label{eq:OM1}
 \left( \sum_{i=1}^n (A_i+B_i)^p \right)^{\frac{1}{p}} \leq \left( \sum_{i=1}^n A_i^p\right)^{\frac{1}{p}} + \left( \sum_{i=1}^n B_i^p\right)^{\frac{1}{p}}\qquad \mbox{for $p\geq 1$}.
\end{equation}

Moreover, let the map $\Phi:\mathcal{B}(\mathscr{H})\oplus \cdots \oplus \mathcal{B}(\mathscr{H}) \mapsto \mathcal{B}(\mathscr{H})$ be defined by
\[
\Phi(A_1\oplus \cdots \oplus A_n)=\frac{1}{n}(A_1+\cdots+A_n),
\]
whence $\Phi$ is a unital positive linear map and the Minkowski operator inequality \eqref{eq:OM1} turns to a more general form  as
\[
\Phi((\Bbb{A}+\Bbb{B})^p)^{\frac{1}{p}} \leq \Phi(\Bbb{A}^p)^{\frac{1}{p}}+\Phi(\Bbb{B}^p)^{\frac{1}{p}} \qquad \mbox{for $p\geq 1$},
\]
where $\Bbb{A}=A_1\oplus \cdots \oplus A_n$ and $\Bbb{B}=B_1\oplus \cdots \oplus B_n$. In that way, convexity of the operator mapping $A\mapsto \Phi(A^p)^{1/p}$ for $A\in \mathcal{B}(\mathscr{H})^{++}$ is equivalent to the Minkowski type  operator inequality
\begin{align}\label{eq:MO}
\Phi\left((A+B)^p\right)^{1/p}\leq \Phi\left(A^p\right)^{1/p}+\Phi\left(B^p\right)^{1/p}.
\end{align}
Thus, we would expect the  Minkowski type operator inequality \eqref{eq:MO} for positive invertible operators $A,B$.

However, by the non-commutativity of operators, we have the following counterexamples: Let
\[
A_1=\begin{pmatrix} 3 & -1 & 0  \\ -1 & 1 & 0 \\ 0 & 0 & 1 \end{pmatrix}, A_2=\begin{pmatrix} 1 & 1 & 0  \\ 1 & 1 & 0 \\ 0 & 0 & 1 \end{pmatrix}, B_1=\begin{pmatrix} 1 & 0 & 0 \\ 0 & 1 & -1 \\ 0 & -1 & 1 \end{pmatrix}, \ \mbox{and} \ B_2=\begin{pmatrix} 1 & 0 & 0 \\ 0 & 1 & 1 \\ 0 & 1 & 2 \end{pmatrix}.
\]
In the case of $n=2$ and $p=2$, an easy computation yields
\begin{align*}
& (A_1^2+A_2^2)^{1/2}+(B_1^2+B_2^2)^{1/2}-\left((A_1+B_1)^2+(A_2+B_2)^2\right)^{1/2}\\
& = \begin{pmatrix} 0.180869 & -0.119435 & -0.238421 \\ -0.119435 & 0.501802 & 0.0713442 \\ -0.238421 & 0.0713442 & 0.188193 \end{pmatrix} \not\geq 0,
\end{align*}
because its eigenvalues are $\{ 0.603875, 0.32367, -0.0562778 \}$. Therefore, the operator version \eqref{eq:OM1} does not hold in general.\par
  We refer the reader to \cite{Brn-Hi2,Crl-Lb,Crl-Lb2} to see nice     studies  of  Minkowski type inequalities  for trace mappings and operator means.

In this section, we present  Minkowski type operator inequalities for a unial positive linear map by using a generalized Kantorovich constant. As an application, we have an estimate for the operator valued determinants of a unital positive linear map by using the Specht ratio.\par

First of all, we recall the definition of the generalized condition number, the generalized Kantorovich constant and the Specht ratio,   see \cite[pp.70--pp.71,Definition 2.2]{FMPS}. Following Turing \cite{Turing}, the condition number $h=h(A)$ of an invertible operator $A$ is defined by $h(A)=\NORM{A} \NORM{A^{-1}}$, where $\NORM{\cdot}$ stands for the operator norm. If a positive invertible operator $A$ satisfies the condition $mI\leq A\leq MI$, then it is thought as $M=\NORM{A}$ and $m=\NORM{A^{-1}}^{-1}$, so that $h=h(A)=\frac{M}{m}$. It is called the generalized condition number. The generalized Kantorovich constant $K(h,p)$ is defined by
\begin{equation} \label{eq:K}
K(h,p)=\frac{h^p-h}{(p-1)(h-1)}\left( \frac{p-1}{p} \frac{h^p-1}{h^p-h}\right)^p \qquad \mbox{for all $p\in {\Bbb R}$}
\end{equation}
and the Specht ratio $S(h)$ is defined by
\begin{equation} \label{eq:S}
S(h)=\frac{(h-1)h^{\frac{1}{h-1}}}{e \log h}\quad (h\not= 1) \qquad \mbox{and}\qquad S(1)=1.
\end{equation}

We mention some important properties of $K(h,p)$ and $S(h)$:
\begin{lemma}{\cite[Theorem 2.54, Theorem 2.56]{FMPS}} \label{lem-KS}
Let $h>0$ be given. The following properties hold:
\begin{enumerate}
\item[(1)] $K(h,p)=K(h^{-1},p)$ for all $p\in {\Bbb R}$.
\item[(2)] $K(h,p)=K(h,1-p)$ for all $p\in {\Bbb R}$.
\item[(3)] $K(h,0)=K(h,1)=1$ and $K(1,p)=1$ for all $p\in {\Bbb R}$.
\item[(4)] $K(h^r,\frac{p}{r})^{\frac{1}{p}}=K(h^p,\frac{r}{p})^{-\frac{1}{r}}$ for all $pr\not= 0$.
\item[(5)] $\lim_{r\to 0} K(h^r,\frac{p}{r})=S(h^p)$.
\end{enumerate}
\end{lemma}

To present the main theorem of this section, we need the following reverse Jensen operator inequalities, see \cite[Lemma 4.3]{FMPS}:
\begin{lemma} \label{lem-RJ}
Let $\Phi:\mathcal{B}(\mathscr{H})\mapsto \mathcal{B}(\mathscr{K})$ be a unital positive linear map, and let $A$ be a positive invertible operator such that $mI\leq A\leq MI$ for some scalars $0<m<M$. Then
\begin{enumerate}
\item[(1)] $K(h,p)\Phi(A)^p \leq \Phi(A^p)\leq \Phi(A)^p$ \quad for $0<p\leq 1$;
\item[(2)]  $\Phi(A)^p \leq \Phi(A^p)\leq K(h,p) \Phi(A)^p$ \quad for $-1\leq p<0$ or $1<p\leq 2$;
\item[(3)]  $K(h,p)^{-1}\Phi(A)^p \leq \Phi(A^p)\leq K(h,p) \Phi(A)^p$ \quad for $p\leq -1$ or $2<p$,
\end{enumerate}
where $h=M/m$ and the generalized Kantorovich constant $K(h,p)$ is defined by \eqref{eq:K}
\end{lemma}
Though the Minkowski type operator inequality \eqref{eq:MO} does not hold in general, by terms of the generalized Kantorovich constant, we obtain the following estimate of the Minkowski type operator inequality:

\begin{theorem} \label{thm-MO}
Let $A$ and $B$ be positive invertible operators such that $\mathrm{Sp}(A),\mathrm{Sp}(B)\subseteq [m,M]$  for some scalars $0<m<M$. If  $\Phi:\mathcal{B}(\mathscr{H})\mapsto \mathcal{B}(\mathscr{K})$ is a unital positive linear map,  then
{\small\begin{equation} \label{eq:MO1}
K(h,p)^{-\frac{1}{p}}\left[ \Phi(A^p)^{\frac{1}{p}}+\Phi(B^p)^{\frac{1}{p}}\right] \leq \Phi((A+B)^p)^{\frac{1}{p}} \leq K(h,p)^{\frac{1}{p}} \left[ \Phi(A^p)^{\frac{1}{p}}+\Phi(B^p)^{\frac{1}{p}}\right]
\end{equation}}
for all $p\geq 1$;
{\small\begin{equation} \label{eq:MO2}
K(h,p)^{\frac{1}{p}}\left[ \Phi(A^p)^{\frac{1}{p}}+\Phi(B^p)^{\frac{1}{p}}\right] \leq \Phi((A+B)^p)^{\frac{1}{p}} \leq K(h,p)^{-\frac{1}{p}} \left[ \Phi(A^p)^{\frac{1}{p}}+\Phi(B^p)^{\frac{1}{p}}\right]
 \end{equation}}
for all $p\leq -1$ or \ $\frac{1}{2}\leq p\leq 1$;
{\small\begin{equation} \label{eq:MO3}
K(h,p)^{\frac{2}{p}}\left[ \Phi(A^p)^{\frac{1}{p}}+\Phi(B^p)^{\frac{1}{p}}\right] \leq \Phi((A+B)^p)^{\frac{1}{p}} \leq K(h,p)^{-\frac{2}{p}} \left[ \Phi(A^p)^{\frac{1}{p}}+\Phi(B^p)^{\frac{1}{p}}\right]
 \end{equation}}
for $-1<p<0$ or \ $0<p<\frac{1}{2}$, where $h=M/m$ and the generalized Kantorovich constant $K(h,p)$ is defined by \eqref{eq:K}.
\end{theorem}

\begin{proof}
First assume that  $p\geq 1$. Then $0<\frac{1}{p}\leq 1$ and  it follows from (1) of Lemma~\ref{lem-RJ} that $\Phi(A^{\frac{1}{p}}) \leq \Phi(A)^{\frac{1}{p}}$. By replacing $A$ by $A^p$ in both sides, we have $\Phi(A)\leq \Phi(A^p)^{\frac{1}{p}}$ and hence
\begin{align}\label{n}
\Phi(A+B) \leq \Phi(A^p)^{\frac{1}{p}}+\Phi(B^p)^{\frac{1}{p}}.
\end{align}
 Since $2mI\leq A+B\leq 2MI$, the generalized condition number of $A+B$ is $h(A+B)=\frac{2M}{2m}=h$ and so we conclude from (1) of Lemma~\ref{lem-RJ} that
\[
K\left(h,1/p\right)\Phi(A+B)^{\frac{1}{p}}\leq \Phi((A+B)^{\frac{1}{p}}).
\]
Replacing $A+B$ by $(A+B)^p$, we have
\[
K(h^p,1/p)\Phi((A+B)^p)^{\frac{1}{p}}\leq \Phi(A+B).
\]
Since $K(h,p)^{\frac{1}{p}}=K(h^p,1/p)^{-1}$ by (4) of Lemma~\ref{lem-KS}, it follows that
\begin{align*}
\Phi((A+B)^p)^{\frac{1}{p}} & \leq K(h^p,1/p)^{-1}\Phi(A+B) \\
& \leq  K(h,p)^{\frac{1}{p}}\left[ \Phi(A^p)^{\frac{1}{p}}+\Phi(B^p)^{\frac{1}{p}}\right],
\end{align*}
where the last inequality comes from \eqref{n}. This gives the second inequality of \eqref{eq:MO1}.\par
Then note that   Lemma~\ref{lem-RJ} and $0<\frac{1}{p}\leq 1$  imply that $\Phi(A+B)\leq \Phi((A+B)^p)^{\frac{1}{p}}$ and $K(h^p,\frac{1}{p})\Phi(A^p)^{\frac{1}{p}}\leq \Phi(A)$, and so
\[
K(h^p,\frac{1}{p})\left[ \Phi(A^p)^{\frac{1}{p}}+\Phi(B^p)^{\frac{1}{p}}\right] \leq \Phi(A)+\Phi(B) = \Phi(A+B) \leq \Phi((A+B)^p)^{\frac{1}{p}}.
\]
Hence we have the first inequality of \eqref{eq:MO1}.\par
Next assume that    $p\leq -1$ or \ $\frac{1}{2}\leq p\leq 1$ so that  $-1\leq \frac{1}{p}<0$ or  $1\leq \frac{1}{p}\leq 2$.  It follows from (2) of Lemma~\ref{lem-RJ} that $\Phi((A+B)^p)^{\frac{1}{p}}\leq \Phi(A+B)$ and $\Phi(A)\leq K(h^p,\frac{1}{p})\Phi(A^p)^{\frac{1}{p}}$, and thus we have the second inequality of \eqref{eq:MO2}:
\begin{align*}
\Phi((A+B)^p)^{\frac{1}{p}}& \leq K(h^p,\frac{1}{p}) \left[ \Phi(A^p)^{\frac{1}{p}}+\Phi(B^p)^{\frac{1}{p}}\right] \\
& = K(h,p)^{-\frac{1}{p}} \left[ \Phi(A^p)^{\frac{1}{p}}+\Phi(B^p)^{\frac{1}{p}}\right].
\end{align*}
Moreover, another use of  (2) of Lemma~\ref{lem-RJ} gives us
\begin{align*}
\Phi(A^p)^{\frac{1}{p}}+\Phi(B^p)^{\frac{1}{p}} & \leq \Phi(A)+\Phi(B) = \Phi(A+B) \\
& \leq K(h^p,\frac{1}{p})\Phi((A+B)^p)^{\frac{1}{p}}\\
& = K(h,p)^{-\frac{1}{p}} \Phi((A+B)^p)^{\frac{1}{p}},
\end{align*}
whence  we have the first inequality of \eqref{eq:MO2}.\par
Finally,  if $-1<p<0$ or \ $0<p<\frac{1}{2}$, then $\frac{1}{p}<-1$ or \ $\frac{1}{p}>2$ and    (3) of Lemma~\ref{lem-RJ} yields  that $\Phi(A^p)^{\frac{1}{p}}\leq K(h^p,\frac{1}{p}) \Phi(A)$  and hence
\[
K(h^p,\frac{1}{p})^{-1}\left[ \Phi(A^p)^{\frac{1}{p}}+\Phi(B^p)^{\frac{1}{p}}\right] \leq \Phi(A)+\Phi(B)=\Phi(A+B).
 \]
Since $\Phi(A+B)\leq K(h^p,\frac{1}{p})\Phi((A+B)^p)^{\frac{1}{p}}$, we have
\[
K(h^p,\frac{1}{p})^{-1}\left[ \Phi(A^p)^{\frac{1}{p}}+\Phi(B^p)^{\frac{1}{p}}\right]\leq K(h^p,\frac{1}{p})\Phi((A+B)^p)^{\frac{1}{p}}
\]
and we get
\[
K(h,p)^{\frac{2}{p}}\left[ \Phi(A^p)^{\frac{1}{p}}+\Phi(B^p)^{\frac{1}{p}}\right] \leq \Phi((A+B)^p)^{\frac{1}{p}}.
\]
This concludes  the first inequality of \eqref{eq:MO3}. Utilizing  (3) of Lemma~\ref{lem-RJ} once more we obtain
 \[
K(h^p,\frac{1}{p})^{-1}\Phi((A+B)^p)^{\frac{1}{p}} \leq \Phi(A)+\Phi(B) \leq K(h^p,\frac{1}{p}) \left[ \Phi(A^p)^{\frac{1}{p}}+\Phi(B^p)^{\frac{1}{p}}\right].
\]
Therefore
\begin{align*}
\Phi((A+B)^p)^{\frac{1}{p}} & \leq K(h^p,\frac{1}{p})^2 \left[ \Phi(A^p)^{\frac{1}{p}}+\Phi(B^p)^{\frac{1}{p}}\right]\\
& = K(h,p)^{-\frac{2}{p}} \left[ \Phi(A^p)^{\frac{1}{p}}+\Phi(B^p)^{\frac{1}{p}}\right]
\end{align*}
from which  we have the second inequality of \eqref{eq:MO3}.\par
Therefore, the proof of Theorem~\ref{thm-MO} is complete.
\end{proof}

As an application of Theorem~\ref{thm-MO}, we have the following complementary inequalities of the Minkowski's operator sum inequalities \eqref{eq:OM1}:
\begin{corollary}
If positive invertible operators $A_1,\ldots,A_n$ and $B_1,\ldots,B_n$ satisfy the condition $mI\leq A_i, B_i \leq MI$ for all $i=1,\ldots,n$ and some scalars $0<m<M$, then
\begin{align*}
 K(h,p)^{-\frac{1}{p}} \left[ \left( \sum_{i=1}^n A_i^p\right)^{\frac{1}{p}} + \left( \sum_{i=1}^n B_i^p\right)^{\frac{1}{p}}\right] & \leq \left( \sum_{i=1}^n (A_i+B_i)^p \right)^{\frac{1}{p}} \\
& \leq K(h,p)^{\frac{1}{p}} \left[ \left( \sum_{i=1}^n A_i^p\right)^{\frac{1}{p}} + \left( \sum_{i=1}^n B_i^p\right)^{\frac{1}{p}}\right]
\end{align*}
for all $p\geq 1$, where $h=M/m$ and the generalized Kantorovich constant $K(h,p)$ is defined by \eqref{eq:K}.
\end{corollary}
\par
\medskip

 Fuglede-Kadison \cite{FK1} and Arveson \cite{Arveson} introduced the normalized determinant for an invertible operator  $A$ in II$_1$-factors with the canonical trace $\tau$:
\[
\Delta_{\tau}(A)=\exp \tau(\log |A|).
\]
Following this, in \cite{FSeo,FIS}, the normalized determinant $\Delta_{\varphi}$ for a positive invertible operator  $A$ and a fixed vector state $\varphi$ is defined by
\[
\Delta_{\varphi}(A) = \exp \varphi(\log A)
\]
as a continuous geometric mean. Along this line, Fujii, Nakamura and Seo \cite{FNS} considered an operator valued determinant $\Delta_{\Phi}$ defined by
\[
\Delta_{\Phi}(A)=\exp \Phi(\log A)
\]
where $\Phi$ is a unital positive linear map. We list some properties of the operator valued determinant: (i) continuity: The map $A\mapsto \Delta_{\Phi}(A)$ is norm continuous. (ii) bounds: $\NORM{A^{-1}}^{-1}\leq \Delta_{\Phi}(A) \leq \NORM{A}$. (iii) power equality: $\Delta_{\Phi}(A^t)=\Delta_{\Phi}(A)^t$ for all real numbers $t$. (iv) homogeneity: $\Delta_{\Phi}(tA)=t\Delta_{\Phi}(A)$ for all positive numbers $t$.\par
 If $A$ and $B$ are positive definite matrices in ${\Bbb M}_n$, then it is known that the following determinant inequalities hold:
\begin{align}\label{det1}
\det(A+B) \geq \det(A) + \det(B)
\end{align}
and
\begin{align}\label{det2}
\det(A+B)^{\frac{1}{n}} \geq \det(A)^{\frac{1}{n}} + \det(B)^{\frac{1}{n}}.
 \end{align}
 The   inequality \eqref{det2} is a Minkowski's determinant inequality,
 see \cite[Theorem 7.8.21]{HJ}. As an application of Theorem~\ref{thm-MO}, we present
 a variant of \eqref{det1} for operator valued determinants, which is
 an estimate of   operator valued determinant.
\begin{corollary} \label{thm-OVD}
Let $\Phi:\mathcal{B}(\mathscr{H})\mapsto \mathcal{B}(\mathscr{K})$ be a unital positive linear map and let $A$ and $B$ be positive invertible operators such that $mI\leq A, B \leq MI$ for some scalars $0<m<M$.   Then
\begin{equation} \label{eq:OVD}
S(h)^{-2}\left[ \Delta_{\Phi}(A)+\Delta_{\Phi}(B) \right] \leq \Delta_{\Phi}(A+B) \leq S(h)^2\left[ \Delta_{\Phi}(A)+\Delta_{\Phi}(B) \right]
\end{equation}
where the Specht ratio $S(h)$ is defined by \eqref{eq:S}.
\end{corollary}

\begin{proof}
It is known that $\Phi(A^p)^{\frac{1}{p}} \mapsto \exp \Phi(\log A)$ as $p\to 0$ for every unital positive linear map $\Phi$ and every positive invertible operator $A$. By (5) of Lemma~\ref{lem-KS}, we have $K(h,p)^{\frac{1}{p}}=K(h^p,\frac{1}{p})^{-1} \to S(h)^{-1}$ as $p\to 0$, whence Corollary~\ref{thm-OVD} follows from Theorem~\ref{thm-MO}.
\end{proof}
\par
\medskip

  The next result   provides an   estimation for operator valued determinant  in the sense of  Minkowski type operator inequality. In particular, it gives a variant of \eqref{det2} for   operator valued determinants.
 \begin{theorem} \label{thm-MOVD}
Let $\Phi:\mathcal{B}(\mathscr{H})\mapsto \mathcal{B}(\mathscr{K})$ be a unital positive linear map and let $A$ and $B$ be positive invertible operators such that $mI\leq A, B \leq MI$ for some scalars $0<m<M$. Put $h=M/m$. Then
\begin{equation} \label{eq:MOVD}
2^{1-\frac{1}{p}}S(h^{\frac{1}{p}})^{-3} K(h,1/p) \Delta_{\Phi}(A+B)^{\frac{1}{p}}\leq
\Delta_{\Phi}(A)^{\frac{1}{p}}+\Delta_{\Phi}(B)^{\frac{1}{p}} \leq 2^{1-\frac{1}{p}} S(h^{\frac{1}{p}})^3   \Delta_{\Phi}(A+B)^{\frac{1}{p}}
\end{equation}
for all $p\geq 1$, where the Specht ratio $S(h)$ is defined by \eqref{eq:S} and the generalized Kantorovich constant $K(h,p)$ is defined by \eqref{eq:K}.
\end{theorem}

\begin{proof}
If $p\geq 1$,  then $0<\frac{1}{p}\leq 1$ and it follows from (1) of Lemma~\ref{lem-RJ} that
\[
K(h,\frac{1}{p}) \left( \frac{A+B}{2}\right)^{\frac{1}{p}} \leq \frac{A^{\frac{1}{p}}+B^{\frac{1}{p}}}{2} \leq \left( \frac{A+B}{2}\right)^{\frac{1}{p}},
\]
whence
\[
2^{1-\frac{1}{p}}\cdot K(h,\frac{1}{p}) (A+B)^{\frac{1}{p}} \leq A^{\frac{1}{p}}+B^{\frac{1}{p}}\leq 2^{1-\frac{1}{p}}(A+B)^{\frac{1}{p}}.
\]
Since $t\mapsto\log t$ is operator monotone, we have
\begin{equation} \label{eq:ab}
\log 2^{1-\frac{1}{p}}\cdot K(h,\frac{1}{p})+\Phi(\log (A+B)^{\frac{1}{p}}) \leq
\Phi(\log(A^{\frac{1}{p}}+B^{\frac{1}{p}}))\leq \log 2^{1-\frac{1}{p}}+\Phi(\log (A+B)^{\frac{1}{p}})
\end{equation}
for every unital positive linear map $\Phi$.   It is known in \cite[Theorem 2.15]{bk-MPS2} that if $X$ is a self-adjoint operator with the condition number $h=M/m$, then $\langle \exp A x,x\rangle\leq S\left(e^{M-m}\right)\exp\langle  A x,x\rangle$ holds for every unit vector $x\in\mathscr{H}$. Noting  that  $(\log 2m^{\frac{1}{p}})I \leq \Phi\left(\log(A^{\frac{1}{p}}+B^{\frac{1}{p}})\right)\leq  (\log 2M^{\frac{1}{p}})I$, this implies that
{\small \begin{align*}
\< \exp \Phi\left(\log(A^{\frac{1}{p}}+B^{\frac{1}{p}})\right)x,x\> & \leq S\left(\exp\left(\log 2M^{\frac{1}{p}}-\log 2m^{\frac{1}{p}}\right)\right) \exp \< \Phi\left(\log(A^{\frac{1}{p}}+B^{\frac{1}{p}})\right)x,x\> \\
& \leq S(h^{\frac{1}{p}}) \exp \< \left[ \log 2^{1-\frac{1}{p}}+\Phi\left(\log (A+B)^{\frac{1}{p}}\right)\right] x,x\> \\
& \qquad \qquad \qquad \qquad \qquad \qquad \mbox{by the second inequality of \eqref{eq:ab}} \\
& \leq S(h^{\frac{1}{p}}) 2^{1-\frac{1}{p}} \< \exp \Phi\left(\log(A+B)^{\frac{1}{p}}\right)x,x\>
\end{align*}}
for every unit vector $x\in \mathscr{H}$, where the last inequality follows from the   Jensen inequality.  Therefore
\[
\Delta_{\Phi}(A^{\frac{1}{p}}+B^{\frac{1}{p}})\leq 2^{1-\frac{1}{p}} S(h^{\frac{1}{p}}) \Delta_{\Phi}((A+B)^{\frac{1}{p}}).
\]
Noting that Lemma \ref{lem-KS} says that $K(h,p)^{2/p}=K(h^p,1/p)^{-2}$ tends to $S(h)^{-2}$, when $p\to 0$,    it follows from \eqref{eq:MO3} of Theorem~\ref{thm-MO} that
\begin{align}\label{eq:K1}
  S(h)^{-2} \left[\Delta_{\Phi}(A)+\Delta_{\Phi}(B)\right]\leq \Delta_{\Phi}(A+B).
\end{align}
Since $\Delta_{\Phi}(A^{\frac{1}{p}})=\Delta_{\Phi}(A)^{\frac{1}{p}}$ by the power equality of $\Delta_{\Phi}$,  replacing $A$ and $B$ in \eqref{eq:K1}, respectively by $A^{1/p}$ and $B^{1/p}$, we deduce
\[
S(h^{\frac{1}{p}})^{-2}\left[ \Delta_{\Phi}(A)^{\frac{1}{p}}+\Delta_{\Phi}(B)^{\frac{1}{p}} \right] \leq
\Delta_{\Phi}(A^{\frac{1}{p}}+B^{\frac{1}{p}})\leq 2^{1-\frac{1}{p}} S(h^{\frac{1}{p}}) \Delta_{\Phi}(A+B)^{\frac{1}{p}}
\]
and thus we have the second inequality of \eqref{eq:MOVD}:
\[
\Delta_{\Phi}(A)^{\frac{1}{p}}+\Delta_{\Phi}(B)^{\frac{1}{p}} \leq 2^{1-\frac{1}{p}} S(h^{\frac{1}{p}})^3 \Delta_{\Phi}(A+B)^{\frac{1}{p}}.
\]
For the first inequality of \eqref{eq:MOVD}, since $(\log(2m)^{\frac{1}{p}})I\leq \Phi(\log(A+B)^{\frac{1}{p}})\leq  (\log (2M)^{\frac{1}{p}})I$, it follows from the first inequality of \eqref{eq:ab} that
\[
\exp \Phi(\log(A+B)^{\frac{1}{p}}) \leq S(h^{\frac{1}{p}})2^{\frac{1}{p}-1}K(h,\frac{1}{p})^{-1} \exp \Phi(\log(A^{\frac{1}{p}}+B^{\frac{1}{p}}))
\]
and thus
\[
\Delta_{\Phi}(A^{\frac{1}{p}}+B^{\frac{1}{p}}) \geq S(h^{\frac{1}{p}})^{-1}2^{1-\frac{1}{p}}K(h,\frac{1}{p}) \Delta_{\Phi}((A+B)^{\frac{1}{p}}).
\]
By Corollary~\ref{thm-OVD} and the power equality of $\Delta_{\Phi}$, we have
\begin{align*}
\Delta_{\Phi}(A)^{\frac{1}{p}}+\Delta_{\Phi}(B)^{\frac{1}{p}} & = \Delta_{\Phi}(A^{\frac{1}{p}})+\Delta_{\Phi}(B^{\frac{1}{p}}) \\
& \geq S(h^{\frac{1}{p}})^{-2} \Delta_{\Phi}(A^{\frac{1}{p}}+B^{\frac{1}{p}}) \\
& \geq S(h^{\frac{1}{p}})^{-3} \cdot 2^{1-\frac{1}{p}}\cdot K(h,\frac{1}{p}) \Delta_{\Phi}(A+B)^{\frac{1}{p}}
\end{align*}
and thus we have the first inequality of \eqref{eq:MOVD}. The  proof is now complete. \par
\end{proof}


 Finally, as an application, we present complementary results of Minkowski type matrix trace inequalities due to Carlen and Lieb \cite{Crl-Lb,Crl-Lb2}, Bekjan \cite{Bkj}, Ando and Hiai \cite{AH}. Carlen and Lieb showed the following Minkowski type trace inequalities: For positive definite matrices $A_1,\ldots, A_n$ and $B_1,\ldots, B_n$,
\begin{equation} \label{eq:tm1}
\Tr\left[\left( \sum_{i=1}^n (A_i+B_i)^p\right)^{1/p}\right] \geq \Tr\left[\left(\sum_{i=1}^n A_i^p\right)^{1/p}\right] + \Tr\left[\left(\sum_{i=1}^n B_i^p\right)^{1/p}\right]
\end{equation}
for $0< p\leq 1$ and
\begin{equation} \label{eq:tm2}
\Tr\left[\left( \sum_{i=1}^n (A_i+B_i)^p\right)^{1/p}\right] \leq \Tr\left[\left(\sum_{i=1}^n A_i^p\right)^{1/p}\right] + \Tr\left[\left(\sum_{i=1}^n B_i^p\right)^{1/p}\right]
\end{equation}
for $1\leq p\leq 2$, and the trace function $\Tr\left[ (\sum_{i=1}^n A_i^p)^{1/p}\right]$ is neither convex nor concave for all $p>2$.\par Firstly, we  present \eqref{eq:tm1} and \eqref{eq:tm2} under  more general setting.
\begin{theorem} \label{thm-T}
Let $A$ and $B$ be positive definite matrices such that $mI\leq A,B\leq MI$ for some scalar $0<m<M$ and $h=M/m$. Let $\Phi$ be a unital positive linear map. Then
\begin{align} \label{eq:T-1}
K(h,p)^{1/p} \left(\Tr\left[ \Phi(A^p)^{1/p}\right]+\Tr\left[ \Phi(B^p)^{1/p}\right]\right) & \leq \Tr\left[ \Phi((A+B)^p)^{1/p}\right] \notag \\
& \leq K(h,p)^{-1/p}\left( \Tr\left[ \Phi(A^p)^{1/p}\right]+\Tr\left[ \Phi(B^p)^{1/p}\right]\right)
\end{align}
for all $p<0, 0<p\leq 1$, and
\begin{align*}
K(h,p)^{-1/p}\left(\Tr\left[ \Phi(A^p)^{1/p}\right]+\Tr\left[ \Phi(B^p)^{1/p}\right]\right) & \leq \Tr\left[ \Phi((A+B)^p)^{1/p}\right] \\
& \leq K(h,p)^{1/p}\left( \Tr\left[ \Phi(A^p)^{1/p}\right]+\Tr\left[ \Phi(B^p)^{1/p}\right]\right)
\end{align*}
for all $p\geq 1$, where $K(h,p)$ is the generalized Kantorovich constnat.
\end{theorem}

\begin{proof}
Suppose that $p<0, 0<p\leq 1$. Since $1/p<0, 1/p\geq 1$, i.e., $f(x)=x^{1/p}$ is convex, we have
\[
\Tr\left[\Phi(A)^{1/p}\right] \leq \Tr\left[\Phi(A^{1/p})\right]
\]
and by replacing $A$ by $A^p$ in both sides, we have
\begin{equation} \label{eq:p1}
\Tr\left[\Phi(A^p)^{1/p}\right] \leq \Tr\left[\Phi(A)\right].
\end{equation}
Hence it follows that
\[
\Tr\left[\Phi(A^p)^{1/p}\right]+\Tr\left[\Phi(B^p)^{1/p}\right]\leq \Tr\left[\Phi(A)\right]+\Tr\left[\Phi(B)\right] = \Tr\left[\Phi(A+B)\right].
\]
By (2) and (3) of Lemma~\ref{lem-RJ}, we have
\[
\Phi((A+B)^{1/p}) \leq K(h,1/p)\Phi(A+B)^{1/p}
\]
and thus
\[
\Phi(A+B) \leq K(h^p,1/p)\Phi((A+B)^p)^{1/p}.
\]
Hence we have
\[
K(h^p,1/p)^{-1} \left(\Tr\left[ \Phi(A^p)^{1/p}\right]+\Tr\left[ \Phi(B^p)^{1/p}\right]\right) \leq \Tr\left[ \Phi((A+B)^p)^{1/p}\right] .
\]
This concludes the first inequality of \eqref{eq:T-1}, since $K(h,p)^{1/p}=K(h^p,1/p)^{-1}$.\par
By \eqref{eq:p1}, we have
\[
\Tr\left[ \Phi((A+B)^p)^{1/p}\right] \leq \Tr\left[ \Phi(A+B)\right].
\]
By (2) and (3) of Lemma~\ref{lem-RJ}, we have
\[
\Phi(A)+\Phi(B) \leq K(h^p,1/p)\left( \Phi(A^p)^{1/p}+\Phi(B^p)^{1/p}\right)
\]
and thus
\[
\Tr\left[ \Phi((A+B)^p)^{1/p}\right] \leq K(h^p,1/p)\left( \Tr\left[\Phi(A^p)^{1/p}\right]+\Tr\left[\Phi(B^p)^{1/p}\right] \right).
\]
This implies the second inequality of \eqref{eq:T-1}.\par
In the case of $p\geq 1$, it follows from \eqref{eq:MO1} of Theorem~\ref{thm-MO}.


\end{proof}

In particular, by Theorem~\ref{thm-T}, we show the reverse inequality of Minkowski type trace ones \eqref{eq:tm1} and \eqref{eq:tm2} for $0<p\leq 2$, and give estimates of the upper and lower bounds of Minkowski type ones for $p\geq 2$:

\begin{corollary}
Let $A$ and $B$ be positive definite matrices such that $mI\leq A,B\leq MI$ for some scalar $0<m<M$ and $h=M/m$. Then
{\small\begin{align*}
&  K(h,p)^{-1/p}\left(\Tr\left[\left(\sum_{i=1}^n A_i^p\right)^{1/p}\right] + \Tr\left[\left(\sum_{i=1}^n B_i^p\right)^{1/p}\right]\right) \leq \Tr\left[\left( \sum_{i=1}^n (A_i+B_i)^p\right)^{1/p}\right] \\
& \leq K(h,p)^{-1/p} \left( \Tr\left[\left(\sum_{i=1}^n A_i^p\right)^{1/p}\right] + \Tr\left[\left(\sum_{i=1}^n B_i^p\right)^{1/p}\right] \right)
\end{align*}}
for all $(-\infty,1]\backslash\{0\}$, and
{\small\begin{align*}
& K(h,p)^{-1/p} \left( \Tr\left[\left(\sum_{i=1}^n A_i^p\right)^{1/p}\right] + \Tr\left[\left(\sum_{i=1}^n B_i^p\right)^{1/p}\right] \right)\\
& \leq \Tr\left[\left( \sum_{i=1}^n (A_i+B_i)^p\right)^{1/p}\right]
\leq K(h,p)^{1/p}\left(\Tr\left[\left(\sum_{i=1}^n A_i^p\right)^{1/p}\right] + \Tr\left[\left(\sum_{i=1}^n B_i^p\right)^{1/p}\right]\right)
\end{align*}}
for all $p\geq1$, where $K(h,p)$ is the generalized Kantorovich constant.
\end{corollary}

\par
\bigskip

\textbf{Acknowledgement.} The first author is  supported by a grant from the Iran National Science Foundation (INSF- No. 97018906).
 The second author is supported by Grant-in-Aid for Scientific Research (C), JSPS KAKENHI Grant Number JP 19K03542.

\par
\medskip


\end{document}